\documentclass{birkau}

\usepackage{amsmath,amssymb,url}

\numberwithin{equation}{section}

\theoremstyle{plain}
\newtheorem{theorem}{Theorem}[section]
\newtheorem{lemma}[theorem]{Lemma}
\newtheorem{proposition}[theorem]{Proposition}

\theoremstyle{definition}
\newtheorem{question}[theorem]{Question}

\DeclareMathOperator{\con}{con}
\DeclareMathOperator{\var}{var}

\begin{document}

\title[Special elements of the lattice of monoid varieties]{Special elements of the lattice\\ of monoid varieties}

\author[S. V. Gusev]{Sergey V. Gusev}
\address{Institute of Natural Sciences and Mathematics,
Ural Federal University,
Lenina str. 51, 620000 Ekaterinburg,
Russia}
\email{sergey.gusb@gmail.com}

\thanks{The work is supported by the Russian Foundation for Basic Research (grant 17-01-00551) and by the Ministry of Education and Science of the Russian Federation (project 1.6018.2017/8.9).}



\begin{abstract}
We completely classify all neutral and costandard elements in the lattice $\mathbb{MON}$ of all monoid varieties. Further, we prove that an arbitrary upper-modular element of $\mathbb{MON}$ except the variety of all monoids is either a completely regular or a commutative variety. Finally, we verify that all commutative varieties of monoids are codistributive elements of $\mathbb{MON}$. Thus, the problems of describing codistributive or upper-modular elements of $\mathbb{MON}$ are completely reduced to the completely regular case.
\end{abstract}

\maketitle

\section{Introduction and summary}
\label{section intr}

The lattice of all semigroup varieties denoted hereinafter by $\mathbb{SEM}$ has been the subject of an intensive research over the last five decades. An extensive and quite diverse material has been accumulated in this direction. It is systematically presented in the survey~\cite{Shevrin-Vernikov-Volkov-09}. In sharp contrast, the lattice $\mathbb{MON}$ of all monoid varieties has received much less attention over the years (when referring to monoid varieties, we consider monoids as algebras with an associative binary operation and the nullary operation that fixes the unit element). Up to the recent time, the latest lattice have been considered in the articles~\cite{Head-68,Pollak-81,Wismath-86} only. However, recently interest in the lattice $\mathbb{MON}$ has grown. This is confirmed by the fact that several papers devoted principally to an examination of identities of monoids contain some non-trivial results about lattices of varieties (see~\cite{Jackson-05,Jackson-Lee-17+,Lee-12,Lee-14}, for instance).

Many questions about the lattice $\mathbb{SEM}$ are formulated in terms of identities. It is proved in the early 1970's in~\cite{Burris-Nelson-71a,Burris-Nelson-71b} that the lattice $\mathbb{SEM}$ does not satisfy any non-trivial identity. A similar fact for the lattice $\mathbb{MON}$ was established quite recently in~\cite{Gusev-17+}. In view of this result, it seems natural to study varieties of monoids with different identities in subvariety lattice. The most important lattice identities are the distributive and modular laws. The problem of describing monoid varieties with distributive subvariety lattice seems to be quite difficult. Indeed, it turns out that even the examination of a much more stronger restriction ``to be a chain'' to lattices of monoid varieties required very considerable efforts (see~\cite{Gusev-Vernikov-17+}).

In this paper, we study several restrictions on the monoid varieties related to the distributive and modular laws. More exactly, we consider special elements of several types in the lattice $\mathbb{MON}$. Let us recall definitions of those of them which will be used below.
An element $x$ of a lattice $L$ is called
\begin{align*}
&\text{\emph{neutral} if}&&\forall\,y,z\in L\colon\ (x\vee y)\wedge(y\vee z)\wedge(z\vee x)\\
&&&\phantom{\forall\,y,z\in L\colon}{}=(x\wedge y)\vee(y\wedge z)\vee(z\wedge x);\\
&\text{\emph{costandard} if}&&\forall\,y,z\in L\colon\quad(x\wedge y)\vee z=(x\vee z)\wedge(y\vee z);\\
&\text{\emph{codistributive} if}&&\forall\,y,z\in L\colon\quad x\wedge(y\vee z)=(x\wedge y)\vee(x\wedge z);\\
&\text{\emph{modular} if}&&\forall\,y,z\in L\colon\quad y\le z\longrightarrow(x\vee y)\wedge z=(x\wedge z)\vee y;\\
&\text{\emph{upper-modular} if}&&\forall\,y,z\in L\colon\quad y\le x\longrightarrow x\wedge(y\vee z)=y\vee(x\wedge z).
\end{align*}
\emph{Lower-modular} elements are defined dually to upper-modular once. It is well known that an element $x\in L$ is neutral if and only if, for any $y,z\in L$, the elements $x$, $y$ and $z$ generate a distributive sublattice of $L$ (see~\cite[Theorem~254]{Gratzer-11}, for instance). Note that neutral elements play an important role in the general theory of lattices. In particular, it is well known that if $a$ is a neutral element in a lattice $L$ then $L$ is decomposable into a subdirect product of the principal ideal and the principal filter of $L$ generated by $a$ (see~\cite[Theorem 254]{Gratzer-11}, for instance). Thus, the knowledge of which elements of a lattice are neutral gives essential information on the structure of the lattice as a whole. It is evident that a neutral element is both lower-modular and costandard; a costandard element is modular; a codistributive element is upper-modular. It is well known also that a costandard element is codistributive (see~\cite[Theorem 253]{Gratzer-11}, for instance). Some information about special elements in arbitrary lattices can be found in~\cite[Section~III.2]{Gratzer-11} and~\cite[Chapter 1]{Seselja-Tepavcevic-01}. 

There are many interesting and deep results about special elements of the mentioned above types in the lattice $\mathbb{SEM}$ (see the surveys~\cite[Section 14]{Shevrin-Vernikov-Volkov-09} and~\cite{Vernikov-15}). In particular, neutral elements of the lattice $\mathbb{SEM}$ have been completely described in~\cite[Proposition 4.1]{Volkov-05}, while in~\cite[Theorem 1.3]{Vernikov-11} it is proved that a semigroup variety is a costandard element of the lattice $\mathbb{SEM}$ if and only if it is a neutral element of this lattice. Codistributive elements of $\mathbb{SEM}$ were examined in~\cite{Vernikov-11}, while upper-modular elements of $\mathbb{SEM}$ were considered in~\cite{Vernikov-08a,Vernikov-08b}.

Special elements in the lattice $\mathbb{MON}$ were not studied so far. The main results of this work give a complete descriptions of neutral and costandard elements of the lattice $\mathbb{MON}$. Besides that, we obtain a valuable information about codistributive and upper-modular elements of $\mathbb{MON}$. 

In order to formulate the first main result of the article, we fix notation for a few concrete varieties. The trivial variety of monoids is denoted by $\bf T$, while $\bf MON$ denotes the variety of all monoids.  We denote by \textbf{SL} the variety of all semilattice monoids. Our first main result is the following theorem.

\begin{theorem}
\label{neutral theorem}
For a monoid variety $\bf V$, the following are equivalent:
\begin{itemize}
\item[\textup{(i)}] $\bf V$ is a modular, lower-modular and upper-modular element of the lattice $\mathbb{MON}$;
\item[\textup{(ii)}] $\bf V$ is a neutral element of the lattice $\mathbb{MON}$;
\item[\textup{(iii)}] $\bf V$ is one of the varieties $\mathbf T$, $\mathbf{SL}$ or $\mathbf{MON}$.
\end{itemize}
\end{theorem}

In order to formulate the second main result of the article, we need some notation. We denote by $F^1$ the free monoid over a countably infinite alphabet. The words (i.e., elements of $F^1$) unlike letters are written in bold. Two sides of identities we connect by the symbol~$\approx$, while the symbol~$=$ denotes the equality relation on $F^1$. For an identity system $\Sigma$, we denote by $\var\,\Sigma$ the variety of monoids given by $\Sigma$. Put 
\[
\mathbf C_n = \var\{x^n\approx x^{n+1}, xy\approx yx\}
\]
where $n\ge 2$. Our second main result is the following theorem.

\begin{theorem}
\label{costandard theorem}
For a monoid variety $\bf V$, the following are equivalent:
\begin{itemize}
\item[\textup{(i)}] $\bf V$ is a modular and upper-modular element of the lattice $\mathbb{MON}$;
\item[\textup{(ii)}] $\bf V$ is a costandard element of the lattice $\mathbb{MON}$;
\item[\textup{(iii)}] $\bf V$ is one of the varieties $\mathbf T$, $\mathbf{SL}$, $\mathbf C_2$ or $\mathbf{MON}$.
\end{itemize}
\end{theorem}

As we have already mentioned above, an element of $\mathbb{SEM}$ is neutral if and only if it is costandard. Theorems~\ref{neutral theorem} and~\ref{costandard theorem} show that in the lattice $\mathbb{MON}$ these properties are not equivalent.

The following assertion gives a valuable information about upper-modu\-lar elements of the lattice $\mathbb{MON}$. A variety of monoids is called \emph{proper} if it differs from the variety of all monoids. A variety of monoids is said to be \emph{completely regular} if it consists of \emph{completely regular monoids}~(i.e., unions of groups).

\begin{proposition}
\label{alternative}
If a proper monoid variety $\bf V$ is an upper-modular element of the lattice $\mathbb{MON}$ then $\mathbf V$ is either commutative or completely regular.
\end{proposition}

This proposition immediately implies that any proper monoid variety that is a codistributive element of $\mathbb{MON}$ also is either commutative or completely regular. To determine codistributive elements in the completely regular case, we need to consider, in particular, periodic group varieties. The lattice of periodic group varieties is modular but not distributive. Therefore, it contains the 5-element modular non-distributive sublattice. It is evident that all three pairwise non-comparable elements of this sublattice are non-codistributive elements of $\mathbb{MON}$. We see that the problem of describing codistributive elements of $\mathbb{MON}$ in the completely regular case is closely related to the problem of describing periodic group varieties with distributive subvariety lattice. The latter problem seems to be extremely difficult (see~\cite[Subsection~11.2]{Shevrin-Vernikov-Volkov-09} for more detailed comments), whence the former problem is extremely difficult too. Fortunately, out of the completely regular case, the problem of describing codistributive and even upper-modular elements of $\mathbb{MON}$ possesses the complete decision. Proposition~\ref{alternative} shows that, to achieve this goal, it suffices to consider commutative varieties. It tuns out that the following assertion is true.

\begin{proposition}
\label{commut are codistr}
Every commutative monoid variety is a codistributive element of the lattice $\mathbb{MON}$.
\end{proposition}

We note that Propositions~\ref{alternative} and~\ref{commut are codistr} play an important role in the proof of Theorems~\ref{neutral theorem} and~\ref{costandard theorem}.

The article consists of four sections. Section~\ref{section preliminaries} contains definitions, notation and auxiliary results. In Section~\ref{section on modular elements} several auxiliary results about modular or lower-modular elements of the lattice $\mathbb{MON}$ are collected. Finally, Section~\ref{section proofs} is devoted to the proofs of Theorems~\ref{neutral theorem} and~\ref{costandard theorem} and Propositions~\ref{alternative} and~\ref{commut are codistr}.

\section{Preliminaries}
\label{section preliminaries}

We start with the fact that is a part of the semigroup folklore (it is noted in~\cite[Section 1.1]{Jackson-Lee-17+}, for instance).

\begin{proposition}
\label{MON sublattice of SEM}
Let $M$ be a monoid. We denote by $\mathbf V$ the monoid variety generated by $M$ and by $\mathbf W$ the semigroup variety generated by this monoid. Then the map $\mathbf V\mapsto \mathbf W$ is an injective homomorphism of the lattice $\mathbb{MON}$ into the lattice $\mathbb{SEM}$.
\end{proposition}

Obviously, the varieties $\bf T$ and $\bf MON$ are neutral elements of the lattice $\mathbb{MON}$. It is proved in~\cite[Proposition 2.4]{Volkov-05} that $\bf SL$ is a neutral element of the lattice $\mathbb{SEM}$. Using Proposition~\ref{MON sublattice of SEM} we have that the following statement is true.

\begin{lemma}
\label{some neutral elements}
The varieties $\bf T$, $\bf SL$ and $\bf MON$ are neutral elements of the lattice $\mathbb{MON}$.
\end{lemma}

The variety of all Abelian groups whose exponent divides $n$ is denoted by $\mathbf A_n$. We note that $\mathbf A_1=\mathbf T$. We need the following result obtained in~\cite{Head-68}.

\begin{lemma}
\label{commutative}
If $\mathbf V$ is a periodic commutative monoid variety then $\mathbf V=\mathbf A_n\vee\mathbf M$ where $n$ is some natural number and $\mathbf M$ is one of the varieties $\mathbf T$, $\mathbf{SL}$ or $\mathbf C_m$ for some $m\ge 2$.
\end{lemma}

The \emph{content} of a word \textbf w, i.e., the set of all letters occurring in $\bf w$ is denoted by $\con({\bf w})$. The following statement is well known and can be easily verified (see~\cite[Lemma~2.1]{Gusev-Vernikov-17+}, for instance).

\begin{lemma}
\label{group variety}
For a monoid variety $\mathbf V$, the following are equivalent:
\begin{itemize}
\item[\textup{a)}] $\mathbf V$ is a group variety;
\item[\textup{b)}] $\mathbf V$ satisfies an identity ${\bf u}\approx {\bf v}$ with $\con({\bf u})\ne \con({\bf v})$;
\item[\textup{c)}] $\mathbf{SL\nsubseteq V}$.
\end{itemize}
\end{lemma}

A letter is called \emph{simple in a word} $\bf w$ if it occurs in $\bf w$ only once. The following fact is well known and may be easily verified.

\begin{lemma}
\label{word problem C_2}
A non-trivial identity ${\bf u}\approx {\bf v}$ holds in the variety ${\bf C}_2$ if and only if $\con(\mathbf u)=\con(\mathbf v)$ and the set of all simple in $\mathbf u$ letters coincides with the set of all simple in $\mathbf v$ letters.
\end{lemma}

Put $\mathbf D = \var\{x^2\approx x^3,x^2y\approx xyx\approx yx^2\}$.

\begin{lemma}[{\rm\cite[Lemma~2.14]{Gusev-Vernikov-17+}}]
\label{non-cr and non-commut}
If a variety of monoids ${\bf V}$ is non-completely regular and non-commutative then ${\bf D}\subseteq {\bf V}$.
\end{lemma}

To avoid a confusion, we note that, in~\cite{Gusev-Vernikov-17+}, the variety $\mathbf D$ is denoted by $\mathbf D_1$, while $\mathbf D$ denotes another variety.

\begin{lemma}[{\rm\cite[Lemma~2.6]{Vernikov-08a}}]
\label{maximal group subvariety}
If $\mathbf V$ is a semigroup variety that satisfies the identity $x^n\approx x^{n+1}$ for some $n$ and $\mathbf G$ is a variety of periodic groups then $\mathbf G$ is the largest group subvariety of $\mathbf G\vee\mathbf X$.
\end{lemma}

\section{On modular and lower-modular elements\\
in the lattice $\mathbb{MON}$}
\label{section on modular elements}

The assertions provided in this section will be used in the proof of Theorems~\ref{neutral theorem} and~\ref{costandard theorem} and Proposition~\ref{alternative}. The following assertion was communicated to the author by M.V.~Volkov. But the proof given above is found by the author.

\begin{lemma}
\label{comm comp-reg}
Let $\bf V$ be a non-commutative completely regular monoid variety. Then
\[
\mathbf C_2\vee(\mathbf D \wedge \mathbf V)= \mathbf C_2 \subset \mathbf D = \mathbf D\wedge(\mathbf C_2 \vee \mathbf V).
\]
In particular, $\bf V$ is not a modular element of the lattice $\mathbb{MON}$ and $\mathbf C_2$ is not a lower-modular element of this lattice.
\end{lemma}

\begin{proof}
It easily follows from Lemmas~4.4 and~4.5(ii) of~\cite{Jackson-05}, for instance, that the subvariety lattice of the variety $\bf D$ is the chain $\mathbf{T\subset SL\subset C}_2\subset\mathbf D$. This fact implies that the variety $\mathbf{SL}$ is a maximal completely regular subvariety of the variety $\mathbf D$. The variety $\mathbf V \wedge \mathbf D$ is completely regular. Hence $\mathbf D \wedge \mathbf V\subseteq\bf SL$. Thus, $\mathbf C_2\vee(\mathbf D \wedge \mathbf V) =\mathbf C_2$. On the other hand, the variety $\mathbf C_2 \vee \mathbf V$ is non-completely regular and non-commutative because $\mathbf V$ is non-commutative and $\mathbf C_2$ is non-completely regular. Then Lemma~\ref{non-cr and non-commut} implies that $\mathbf D \subseteq \mathbf V \vee \mathbf C_2$, whence $\mathbf D = \mathbf D\wedge (\mathbf C_2 \vee \mathbf V)$.
\end{proof}

A word \textbf w is called an \emph{isoterm} for a class of monoids if each monoid in the class does not satisfy any non-trivial identity of the form $\mathbf w\approx\mathbf w'$. The empty word~(i.e., the unit element of the monoid $F^1$) is denoted by $\lambda$. Let us fix notation for the following two words: 
\[
\mathbf s=yxyzxz\ \text{ and }\ \mathbf t=yxzxyxz.
\]
Put $\mathbf B_{2,3} = \var\{x^2\approx x^3\}$ and $\mathbf Q = \var\{\mathbf s\approx \mathbf t\}$. Clearly, $\mathbf Q \subset \mathbf B_{2,3}$.

\begin{lemma}
\label{V is commutative and non-combinatorial}
If $\mathbf V$ is a commutative monoid variety containing a non-trivial group then
\[
\mathbf Q\vee(\mathbf B_{2,3} \wedge \mathbf V) \subset \mathbf B_{2,3}\wedge(\mathbf Q \vee \mathbf V).
\]
In particular, $\bf V$ is not a modular element of the lattice $\mathbb{MON}$ and $\mathbf Q$ is not a lower-modular element of this lattice.
\end{lemma}

\begin{proof}
In view of Lemma~\ref{word problem C_2}, $\mathbf C_2\subseteq \mathbf Q$. Since is $\bf V$ is commutative, we have that $\mathbf B_{2,3}\wedge \mathbf V\subseteq \mathbf C_2$. Therefore,
\[
\mathbf Q\vee(\mathbf B_{2,3}\wedge \mathbf V)=\mathbf Q\subseteq \mathbf B_{2,3}\wedge(\mathbf Q\vee \mathbf V).
\]
We are going to verify that this inclusion is strict. It suffices to establish that $\mathbf B_{2,3}\wedge(\mathbf Q \vee \mathbf V)$ violates $\mathbf s\approx \mathbf t$. If, otherwise, $\mathbf B_{2,3}\wedge(\mathbf Q \vee \mathbf V)$ satisfies $\mathbf s\approx \mathbf t$ then there is a sequence of pairwise distinct words $\mathbf w_0,\mathbf w_1,\ldots, \mathbf w_k$ such that $\mathbf w_0 = \mathbf s$, $\mathbf w_k = \mathbf t$ and the identity $\mathbf w_i\approx \mathbf w_{i+1}$ holds either in $\mathbf B_{2,3}$ or in $\mathbf Q \vee \mathbf V$ for every $0\le i<k$. We note that the identity $\mathbf s\approx \mathbf w_1$ does not hold in the variety $\mathbf B_{2,3}$ because $\mathbf s$  is an isoterm for $\mathbf B_{2,3}$. Thus, this identity holds in $\mathbf Q \vee \mathbf V$. In particular, it holds in \textbf Q, whence there exists a \emph{deduction} of the identity $\mathbf s\approx \mathbf w_1$ from the identity $\mathbf s\approx \mathbf t$, i.e., a sequence of pairwise distinct words
\begin{equation}
\label{sequence of words}
\mathbf v_0,\mathbf v_1,\ldots, \mathbf v_m
\end{equation}
such that $\mathbf v_0 = \mathbf s$, $\mathbf v_m = \mathbf w_1$  and, for any $0\le i<m$, there exist words $\mathbf a_i$, $\mathbf b_i$ and endomorphism $\xi_i$ of $F^1$ such that ether $\mathbf v_i=\mathbf a_i\xi_i(\mathbf s)\mathbf b_i$ and $\mathbf v_{i+1}=\mathbf a_i\xi_i(\mathbf t)\mathbf b_i$ or $\mathbf v_i=\mathbf a_i\xi_i(\mathbf t)\mathbf b_i$ and $\mathbf v_{i+1}=\mathbf a_i\xi_i(\mathbf s)\mathbf b_i$. We can assume without loss of generality that the sequence~\eqref{sequence of words} is the shortest deduction of the identity $\mathbf s\approx \mathbf w_1$ from $\mathbf s\approx \mathbf t$. 

Let $\eta$ be an arbitrary endomorphism of the monoid $F^1$. If $\eta$ maps at least one of the letters $x$, $y$ or $z$ into the empty word then the words $\eta(\mathbf s)$ and $\eta(\mathbf t)$ have one of the forms given in Table~\ref{form of words}. We see that in all cases $\eta(\mathbf s)$ and $\eta(\mathbf t)$ contain a subword of the form $\mathbf w^2$. We note that the words $\mathbf s$ and $\mathbf t$ are square-free. This fact and the information collected in Table~\ref{form of words} imply that if the equality $\mathbf a=\mathbf c\eta(\mathbf b)\mathbf d$ holds where $\mathbf a,\mathbf b\in\{\mathbf s,\mathbf t\}$, $\mathbf c,\mathbf d\in F^1$ and $\eta$ is an endomorphism of $F^1$ that maps at least one of the letters $x$, $y$ and $z$ into the empty word then $\eta(\mathbf b)=\lambda$. This fact will be repeatedly used below to obtain a contradiction.

{\tabcolsep=4pt
\begin{table}[tbh]
\caption{The forms of the words $\eta(\mathbf s)$ and  $\eta(\mathbf t)$ depending on the endomorphism $\eta$}
\begin{tabular}{|c|c|c|c|}
\hline
&\multicolumn{3}{c|}{The form of the word if}\\
\cline{2-4}
The word&\multicolumn{1}{c|}{$\eta(x)=\lambda$}&\multicolumn{1}{c|}{$\eta(y)=\lambda$}&\multicolumn{1}{c|}{$\eta(z)=\lambda$}\\
\hline
$\eta(\mathbf s)$&$\mathbf p^2\mathbf q^2$&$(\mathbf p\mathbf q)^2$&$(\mathbf p\mathbf q)^2$\\
\hline
$\eta(\mathbf t)$&$(\mathbf p\mathbf q)^2$&$\mathbf p\mathbf q\mathbf p^2\mathbf q$&$\mathbf p\mathbf q^2\mathbf p\mathbf q$\\
\hline
\end{tabular}
\label{form of words}
\end{table}
}

Suppose first that $\mathbf s=\mathbf v_0=\mathbf a_0\xi_0(\mathbf s)\mathbf b_0$ and $\mathbf v_1=\mathbf a_0\xi_0(\mathbf t)\mathbf b_0$. If the words $\mathbf a_0$ and $\mathbf b_0$ are empty then $\xi_0(a)=a$ for each $a\in\{x,y,z\}$. Then $\mathbf v_1=\mathbf t$. Suppose now that at least one of the words $\mathbf a_0$ and $\mathbf b_0$ is non-empty. Then the endomorphism $\xi_0$ maps one of the letters $x$, $y$ and $z$ into the empty word. We have verified in the previous paragraph that $\xi_0(\mathbf s)=\xi_0(\mathbf t)=\lambda$, whence $\mathbf v_0=\mathbf v_1=\mathbf a_0\mathbf b_0$ in this case. We obtain a contradiction with the fact that the words $\mathbf v_0$ and $\mathbf v_1$ are distinct. 

Suppose now that $\mathbf s=\mathbf v_0=\mathbf a_0\xi_0(\mathbf t)\mathbf b_0$ and $\mathbf v_1=\mathbf a_0\xi_0(\mathbf s)\mathbf b_0$. Since the length of $\mathbf s$ is less than the length of $\mathbf t$, the endomorphism $\xi_0$ maps one of the letters $x$, $y$ and $z$ into the empty word. As we have seen above, $\xi_0(\mathbf t)=\lambda$ in this case. This contradicts with the inequality $\mathbf v_0\ne\mathbf v_1$. Thus, we have verified that $\mathbf v_1=\mathbf t$.

Suppose that $\mathbf t=\mathbf v_1=\mathbf a_1\xi_1(\mathbf s)\mathbf b_1$ and $\mathbf v_2=\mathbf a_1\xi_1(\mathbf t)\mathbf b_1$. Note that the number of occurrences of the letter $x$ in $\xi_1(\mathbf s)$ is not equal to $3$, whence $x\in\con(\mathbf a_1\mathbf b_1)$. Since neither the first nor the last letter of the word $\bf t$ does not coincide with $x$, the length of the word $\mathbf a_1\mathbf b_1$ is more than $1$. Then the endomorphism $\xi_1$ maps one of the letters $x$, $y$ and $z$ into the empty word. Then $\xi_1(\mathbf s)=\lambda$, whence $\mathbf v_1=\mathbf v_2$, a contradiction.

Suppose now that $\mathbf t=\mathbf v_1=\mathbf a_1\xi_1(\mathbf t)\mathbf b_1$ and $\mathbf v_2=\mathbf a_1\xi_1(\mathbf s)\mathbf b_1$. If the words $\mathbf a_1$ and $\mathbf b_1$ are empty then $\xi_1(a)=a$ for each $a\in\{x,y,z\}$. Then $\mathbf v_2=\mathbf s$. But this is impossible because the sequence~\eqref{sequence of words} is the shortest deduction of the identity $\mathbf s\approx \mathbf w_1$ from the identity $\mathbf s\approx \mathbf t$. So, at least one of the words $\mathbf a_1$ and $\mathbf b_1$ is non-empty. Then the endomorphism $\xi_1$ maps one of the letters $x$, $y$ and $z$ into the empty word. Then $\xi_1(\mathbf t)=\lambda$, whence $\mathbf v_1=\mathbf v_2$. We obtain a contradiction with the fact that the words $\mathbf v_2$ and $\mathbf v_1$ are distinct. Thus, we have proved that $m=1$ and $\mathbf v_1=\mathbf w_1=\mathbf t$. Therefore, the identity $\mathbf s\approx \mathbf t$ holds in the variety $\mathbf Q\vee \mathbf V$. Then this variety satisfies the identity $x^2\approx x^3$. But this is impossible because $\mathbf V$ contains a non-trivial group.
\end{proof}

Put $\mathbf F= \var\{xyx\approx xyx^2,x^2y^2\approx y^2x^2,x^2y\approx x^2yx, xytxy\approx yxtxy\}$. Clearly, $\mathbf F \subset \mathbf B_{2,3}$.

\begin{lemma}
\label{C_n whenever n>2}
If $n>2$ then
\[
(\mathbf C_n \wedge \mathbf B_{2,3})\vee \mathbf F \subset (\mathbf C_n \vee \mathbf F)\wedge \mathbf B_{2,3}.
\]
In particular, $\mathbf C_n$ with $n>2$ is not a modular element of the lattice $\mathbb{MON}$ and $\bf F$ is not a lower-modular element of this lattice.
\end{lemma}

\begin{proof}
Evidently, $(\mathbf C_n \wedge \mathbf B_{2,3})\vee \mathbf F \subseteq (\mathbf C_n \vee \mathbf F)\wedge \mathbf B_{2,3}$. We are going to verify that this inclusion is strict. Lemma~\ref{word problem C_2} implies that $\mathbf C_2\subseteq \mathbf F$. Then
\[
(\mathbf C_n \wedge \mathbf B_{2,3})\vee \mathbf F = \mathbf C_2\vee \mathbf F=\mathbf F.
\]
Thus, we need to verify that $\mathbf F\subset (\mathbf C_n \vee \mathbf F)\wedge \mathbf B_{2,3}$. It suffices to establish that $(\mathbf C_n \vee \mathbf F)\wedge \mathbf B_{2,3}$ violates $xyx\approx xyx^2$. If, otherwise, $(\mathbf C_n \vee \mathbf F)\wedge \mathbf B_{2,3}$ satisfies $xyx\approx xyx^2$ then there is a sequence of pairwise distinct words~\eqref{sequence of words} such that $\mathbf v_0 = xyx$, $\mathbf v_m = xyx^2$ and, for any $0\le i<m$, the identity $\mathbf v_i\approx \mathbf v_{i+1}$ holds either in $\mathbf B_{2,3}$ or in $\mathbf C_n \vee \mathbf F$. We note that the identity $xyx\approx \mathbf v_1$ does not hold in the variety $\mathbf B_{2,3}$ because $xyx$ is an isoterm for $\mathbf B_{2,3}$. On the other hand, the identity $xyx\approx \mathbf v_1$ holds in $\mathbf C_n$ if and only if this identity follows from commutativity. Then $\mathbf v_1\in\{x^2y,yx^2\}$. Put 
\[
\mathbf E= \var\{x^2\approx x^3,x^2y\approx xyx,x^2y^2\approx y^2x^2\}.
\]
It is evident that $\mathbf E\subseteq \mathbf F$. Comparison of Propositions 4.2 and 6.9(i) of the article~\cite{Gusev-Vernikov-17+} shows that this inclusion is strict. Thus,
\begin{equation}
\label{E subset F}
\mathbf E \subset \mathbf F
\end{equation}
(we note that in~\cite{Gusev-Vernikov-17+} the variety $\bf F$ is denoted by $\mathbf F_1$). If $\mathbf v_1=x^2y$ then $\mathbf F$ satisfies the identity $x^2y\approx xyx$. We obtain a contradiction with~\eqref{E subset F}. If $\mathbf v_1=yx^2$ then $\mathbf F$ satisfies the identities
\[
xyx\approx \mathbf v_1=yx^2\approx yx^3\approx xyx^2\approx x^2yx\approx x^2y,
\]
and we have a contradiction with~\eqref{E subset F} again. So, $(\mathbf C_n \vee \mathbf F)\wedge \mathbf B_{2,3}$ violates $xyx\approx xyx^2$, whence $\mathbf F\subset (\mathbf C_n \vee \mathbf F)\wedge \mathbf B_{2,3}$.
\end{proof}

\section{Proofs of the main results}
\label{section proofs}

\begin{proof}[Proof of Proposition~\ref{alternative}]
Let $\bf V$ be a proper non-commutative non-complete\-ly regular monoid variety that is an upper-modular element of the lattice $\mathbb{MON}$. Then $\mathbf D\subseteq \mathbf V$ by Lemma~\ref{non-cr and non-commut}. It is proved in~\cite[Lemma 2.16]{Vernikov-08a} that the variety of all semigroups is generated by all minimal non-Abelian varieties of groups. This fact and Proposition~\ref{MON sublattice of SEM} imply that there exists a minimal non-Abelian group variety $\bf G$ such that $\bf G\nsubseteq V$. Then $\mathbf V \wedge \mathbf G=\mathbf A_n$ for some positive integer $n$, whence $ \mathbf C_2\vee(\mathbf V \wedge \mathbf G)=\mathbf C_2\vee\mathbf A_n$. On the other hand, $\mathbf D \subseteq \mathbf C_2 \vee \mathbf G$ by Lemma~\ref{comm comp-reg}. Taking into account the fact that $\mathbf D\subseteq \mathbf V$, we have that $\mathbf D \subseteq \mathbf V\wedge(\mathbf C_2 \vee \mathbf G)$. The variety $\mathbf A_n\vee\mathbf C_2$ is commutative, while $\bf D$ is non-commutative, whence
\[
\mathbf C_2\vee(\mathbf V \wedge \mathbf G) \ne \mathbf V\wedge(\mathbf C_2 \vee \mathbf G).
\] 
Since $\mathbf C_2\subset\mathbf D\subseteq \mathbf V$, we obtain a contradiction with the fact that the variety $\bf V$ is an upper-modular element of the lattice $\mathbb{MON}$.
\end{proof}

\begin{proof}[Proof of Proposition~\ref{commut are codistr}]
Let $\mathbf V$ be a commutative monoid variety and $\mathbf Y$, $\mathbf Z$ be arbitrary monoid varieties. Put $\mathbf X=\mathbf V\wedge(\mathbf Y \vee \mathbf Z)$ and $\mathbf W=(\mathbf V \wedge \mathbf Y) \vee (\mathbf V \wedge \mathbf Z)$. Evidently, $\mathbf W \subseteq \mathbf X$. We need to verify that $\mathbf X \subseteq \mathbf W$. If $\mathbf V\subseteq \mathbf Y$ then
\[
\mathbf X=\mathbf V\wedge(\mathbf Y \vee \mathbf Z)=\mathbf V = \mathbf V \vee (\mathbf V \wedge \mathbf Z) = (\mathbf V \wedge \mathbf Y) \vee (\mathbf V \wedge \mathbf Z)=\mathbf W,
\]
and we are done. Therefore, we may assume that $\mathbf V\nsubseteq \mathbf Y$. By symmetry, $\mathbf V\nsubseteq \mathbf Z$. If $\mathbf V$ is periodic then $\mathbf X$ is periodic too. If $\mathbf V$ is non-periodic then $\mathbf V$ is the variety of all commutative monoids. Since $\mathbf V\nsubseteq \mathbf Y$ and $\mathbf V\nsubseteq \mathbf Z$, the varieties $\mathbf Y$ and $\mathbf Z$ are periodic, whence $\mathbf Y\vee \mathbf Z$ is periodic too. Thus, $\mathbf X$ is a periodic commutative variety. Then Lemma~\ref{commutative} imply that $\mathbf X=\mathbf M\vee\mathbf A_s$ for some $s$ where $\mathbf M$ is one of the varieties $\mathbf T$, $\mathbf{SL}$ or $\mathbf C_n$ with $n\ge 2$. Evidently, $\mathbf M\subseteq \mathbf V$ and $\mathbf M\subseteq \mathbf Y\vee \mathbf Z$. Now we are going to prove that either $\mathbf M\subseteq \mathbf Y$ or $\mathbf M\subseteq \mathbf Z$. If $\mathbf M=\mathbf T$ then we are done. If $\mathbf M=\mathbf{SL}$ then the required fact follows from Lemma~\ref{group variety}. Let now $\mathbf M=\mathbf C_n$ with $n\ge 2$. Suppose that $\mathbf M\nsubseteq \mathbf Y$ and $\mathbf M\nsubseteq \mathbf Z$. It is proved in~\cite[Lemma 2.5]{Gusev-Vernikov-17+} that if a monoid variety does not contain $\mathbf C_n$ then this variety satisfies the identity $x^{n-1}\approx x^{n-1+\ell}$ for some natural $\ell$. This fact implies that there are natural numbers $i$ and $j$ such that $x^{n-1}\approx x^{n-1+i}$ holds in $\mathbf Y$ and $x^{n-1}\approx x^{n-1+j}$ holds in $\mathbf Z$. Then the variety $\mathbf Y\vee \mathbf Z$ satisfies the identity $x^{n-1}\approx x^{n-1+ij}$. We obtain a contradiction with the fact that $\mathbf M\subseteq \mathbf Y\vee \mathbf Z$. Thus, we have proved that either $\mathbf M\subseteq \mathbf Y$ or $\mathbf M\subseteq \mathbf Z$.  Since $\mathbf M\subseteq \mathbf V$, we get that either $\mathbf M\subseteq \mathbf V\wedge \mathbf Y$ or $\mathbf M\subseteq \mathbf V\wedge \mathbf Z$. Therefore, $\mathbf M\subseteq \mathbf W$.

Now we are going to verify that $\mathbf A_s\subseteq \mathbf W$. We note also that $\mathbf W$ is a periodic commutative variety. Then Lemma~\ref{commutative} applies with the conclusion that $\mathbf W = \mathbf M'\vee \mathbf A_r$ for some $r$ where $\mathbf M'$ is one of the varieties $\mathbf T$, $\mathbf{SL}$ or $\mathbf C_n$ with $n\ge 2$. Suppose that $s$ does not divide $r$. Then there exist a prime number $p$ and a positive integer $k$ such that $p^k$ divides $s$ but does not divide $r$. Put $q=p^k$. Then $\mathbf A_q\subseteq \mathbf A_s\subseteq \mathbf X$ but $\mathbf A_q\nsubseteq \mathbf A_r$. It is easy to see that any group subvariety of the variety $\mathbf W=\mathbf M'\vee \mathbf A_r$ is contained in $\mathbf A_r$ (this follows from Proposition~\ref{MON sublattice of SEM} and Lemma~\ref{maximal group subvariety}, for instance). Therefore, $\mathbf A_q\nsubseteq \mathbf W$. Since $\mathbf A_q\subseteq \mathbf X$, we have that $\mathbf A_q\subseteq \mathbf Y\vee\mathbf Z$. It is proved in~\cite[Theorem 1.2]{Vernikov-11} that every variety of periodic Abelian groups is a codistributive element of the lattice $\mathbb{SEM}$. This fact and Proposition~\ref{MON sublattice of SEM} imply that
\[
\mathbf A_q=\mathbf A_q\wedge(\mathbf Y\vee\mathbf Z)=(\mathbf A_q\wedge\mathbf Y)\vee (\mathbf A_q\wedge\mathbf Z).
\]
Since the subvariety lattice of the variety $\mathbf A_q$ is a chain, we have that $\mathbf A_q$ coincides with one of the varieties $\mathbf A_q\wedge \mathbf Y$ or $\mathbf A_q\wedge \mathbf Z$, whence either $\mathbf A_q\subseteq \mathbf Y$ or $\mathbf A_q\subseteq \mathbf Z$. Taking into account that $\mathbf A_q\subseteq \mathbf X\subseteq \mathbf V$, we obtain a contradiction with the fact that $\mathbf A_q\nsubseteq \mathbf W$. Thus, $s$ divides $r$. Then $\mathbf A_s\subseteq \mathbf A_r\subseteq \mathbf W$. Therefore, $\mathbf X = \mathbf M\vee \mathbf A_s\subseteq \mathbf W$. Proposition~\ref{commut are codistr} is proved.
\end{proof}

Propositions~\ref{alternative} and~\ref{commut are codistr} completely reduce the problems of describing codistributive or upper-modular elements in $\mathbb{MON}$ to consideration of completely regular varieties. As we have noted in Section~\ref{section intr}, there are non-codistributive elements of $\mathbb{MON}$ among periodic group varieties. But the following question is open.

\begin{question}
\label{all cr are umod?}
Is every completely regular monoid variety an upper-modular element of $\mathbb{MON}$?
\end{question}

The lattice $\mathbb{SEM}$ contains upper modular but not codistributive elements (see~\cite[Subsection~3.9]{Vernikov-15}). On the other hand, in the lattice of all commutative semigroup varieties the properties to be upper-modular and codistributive elements are equivalent~\cite[Theorem~1.1]{Vernikov-17}. The following question is open so far.

\begin{question}
\label{codistr=umod?}
Does there exist an upper-modular but not codistributive element of the lattice $\mathbb{MON}$?
\end{question}

Clearly, the negative answer to Question~\ref{codistr=umod?} immediately implies the negative answer to Question~\ref{all cr are umod?}.

\begin{proof}[Proof of Theorem~\ref{costandard theorem}]
The implication (ii)~$\Rightarrow$~(i) is obvious. It remains to prove the implications (i)~$\Rightarrow$~(iii) and (iii)~$\Rightarrow$~(ii).

(i) $\Rightarrow$ (iii). Let $\mathbf V$ be a proper monoid variety that is a modular and upper-modular element of the lattice $\mathbb{MON}$. Then Proposition~\ref{alternative} implies that $\mathbf V$ is either completely regular or commutative. The case when $\bf V$ is completely regular and non-commutative is impossible by Lemma~\ref{comm comp-reg}. Therefore, $\bf V$ is commutative. Further, Lemma~\ref{V is commutative and non-combinatorial} implies that all groups of $\bf V$ are trivial. Then $\bf V$ satisfies the identity $x^n\approx x^{n+1}$ for some $n$, whence $\mathbf V\subseteq \mathbf C_n$. By Lemma~\ref{commutative} (see also~\cite[Proposition 5.1]{Gusev-Vernikov-17+}, for instance), $\bf V$ coincides with one of the varieties $\bf T$, $\bf SL$ or $\mathbf C_k$ for some $k\le n$. Finally, Lemma~\ref{C_n whenever n>2} implies that the case $\mathbf V=\mathbf C_k$ with $k\ge 3$ is impossible.

(iii) $\Rightarrow$ (ii). In view of Lemma~\ref{some neutral elements}, the varieties $\mathbf T$, $\mathbf{SL}$ and $\mathbf{MON}$ are neutral elements of the lattice $\mathbb{MON}$. Then these varieties are costandard elements of $\mathbb{MON}$ as well. It remains to prove that the variety $\mathbf C_2$ is a costandard element of this lattice too.

It is easy to verify that an element of a lattice is costandard whenever it is modular and codistibutive. This claim readily follows from~\cite[Theorem~253]{Gratzer-11} or~\cite[Proposition 1.7]{Seselja-Tepavcevic-01}, for instance (to avoid a confusion, we note that, in~\cite{Seselja-Tepavcevic-01}, modular elements are called $s$-\emph{modular} once, while ``a modular element'' means the same as ``a lower-modular element'' in our terminology). In view of the mentioned fact and Proposition~\ref{commut are codistr}, it suffices to prove that $\mathbf C_2$ is a modular element of the lattice $\mathbb{MON}$. Assume the contrary. Then~\cite[Proposition 2.1]{Jezek-81} implies that there exist the varieties $\mathbf U$ and $\mathbf W$ such that $\mathbf U\subset \mathbf W$, $\mathbf U\wedge \mathbf C_2=\mathbf W\wedge \mathbf C_2$ and $\mathbf U\vee \mathbf C_2=\mathbf W\vee \mathbf C_2$. If $\mathbf C_2\subseteq \mathbf U$ then $\mathbf W\wedge \mathbf C_2=\mathbf U\wedge \mathbf C_2=\mathbf C_2$, whence $\mathbf C_2\subseteq \mathbf W$. But then $\mathbf U=\mathbf C_2\vee \mathbf U=\mathbf C_2\vee \mathbf W=\mathbf W$. This contradicts the choice of $\mathbf U$ and $\mathbf W$. Thus, $\mathbf C_2\nsubseteq \mathbf U$. Analogously, $\mathbf C_2\nsubseteq \mathbf W$. It is proved in~\cite[Corollary~2.6]{Gusev-Vernikov-17+} that a monoid variety ${\bf X}$ is completely regular if and only if $\mathbf C_2 \nsubseteq {\bf X}$. Therefore, the varieties $\mathbf U$ and $\mathbf W$ are completely regular.

Suppose that $\mathbf U$ is a group variety. Then $\mathbf{SL} \nsubseteq \mathbf U$. If $\mathbf W$ is a non-group variety then $\mathbf{SL} \subseteq \mathbf W$ by Lemma~\ref{group variety}. Then $\mathbf U\wedge \mathbf C_2=\mathbf T$ but $\mathbf{SL}\subseteq\mathbf W\wedge \mathbf C_2$. We obtain a contradiction with the equality $\mathbf U\wedge \mathbf C_2=\mathbf W\wedge \mathbf C_2$. Therefore, $\mathbf W$ is a group variety. Proposition~\ref{MON sublattice of SEM} and Lemma~\ref{maximal group subvariety} imply that $\mathbf U$ is the largest group subvariety of the variety $\mathbf U\vee\mathbf C_2$. But this is impossible because $\bf W$ is a group variety and
\[
\mathbf U\subset\mathbf W\subset \mathbf W\vee\mathbf C_2=\mathbf U\vee\mathbf C_2.
\]
We see that $\mathbf U$ is a non-group variety. Then $\mathbf{SL}\subseteq \mathbf U$ by Lemma~\ref{group variety}. In this case $\mathbf{SL}\subseteq \mathbf U\wedge \mathbf C_2=\mathbf W\wedge \mathbf C_2\subseteq \mathbf W$. Therefore, $\mathbf W$ is also a non-group variety. Since $\mathbf U$ is completely regular, it satisfies $x\approx x^{n+1}$ for some positive integer $n$. Let $n$ be the least number with such a property, while $\Sigma$ be an identity basis of the variety $\mathbf U$. We denote by $\zeta$ the endomorphism of the monoid $F^1$ which maps each letter $x$ into the word $x^{n+1}$. Put
\[
\Sigma^\ast=\{\zeta(\mathbf u)\approx \zeta(\mathbf v)\mid \mathbf u\approx\mathbf v\in \Sigma\}.
\]
Obviously, $\mathbf U = \var\{x\approx x^{n+1},\Sigma^\ast\}$. If $\mathbf p \approx \mathbf q\in\Sigma^\ast$ then $\con(\mathbf p)=\con(\mathbf q)$ by Lemma~\ref{group variety}. According to Lemma~\ref{word problem C_2}, the variety $\mathbf C_2$ satisfies the identity system $\Sigma^\ast$. Taking into account that $\mathbf U\vee \mathbf C_2=\mathbf W\vee \mathbf C_2$, we obtain that $\mathbf W$ satisfies the identity system $\Sigma^\ast$ too. Since the identity $x^2\approx x^3$ holds in $\mathbf C_2$ and the identity $x\approx x^{n+1}$ holds in $\mathbf U$, the variety $\mathbf U\vee \mathbf C_2=\mathbf W\vee \mathbf C_2$ satisfies $x^2\approx x^{n+2}$. Taking into account that $\mathbf W$ is completely regular, we get that $x\approx x^{n+1}$ holds in $\mathbf W$. Then $\mathbf W\subseteq \mathbf U$. We obtain a contradiction with the choice of the varieties $\mathbf U$ and $\mathbf W$. Thus, we have proved that $\mathbf C_2$ is a modular, and therefore, a costandard element of the lattice $\mathbb{MON}$.
\end{proof}

\begin{proof}[Proof of Theorem~\ref{neutral theorem}]
The implication (iii)~$\Rightarrow$~(ii) follows from Lemma~\ref{some neutral elements}, while the implication (ii~$\Rightarrow$~(i) is obvious. It remains to prove the implication (i)~$\Rightarrow$~(iii). Let $\bf V$ be a modular, lower-modular and upper-modular element of the lattice $\mathbb{MON}$. Theorem~\ref{costandard theorem} implies that $\bf V$ coincides with one of the varieties $\mathbf T$, $\mathbf{SL}$, $\mathbf C_2$ or $\mathbf{MON}$. Lemma~\ref{comm comp-reg} implies that $\mathbf{V\ne C}_2$, and we are done.
\end{proof}

\subsection*{Acknowledgments}

The author is sincerely grateful to Professor B.M.~Vernikov for his great assistance in the improvement of the initial version of the manuscript and to Professor M.V.~Volkov for helpful discussions.

\end{document}